\documentclass[letterpaper, 10 pt, conference]{ieeeconf}  
\IEEEoverridecommandlockouts                              
\usepackage{cite}
\usepackage{amsmath,amssymb,amsfonts}
\usepackage{amsthm}
\usepackage{graphicx}
\usepackage{textcomp}
\usepackage{booktabs}

\theoremstyle{plain}
\newtheorem{theorem}{Theorem}
\newtheorem{lemma}{Lemma}
\newtheorem{proposition}{Proposition}
\theoremstyle{definition}
\newtheorem{assumption}{Assumption}
\theoremstyle{remark}

\title{\LARGE \bf
Advances in Continuous-Time Multiparametric Optimal Control --- The Role of Time-Varying Active Constraints
}

\author{Lida Lamakani and Efstratios N. Pistikopoulos
\thanks{This work was supported by the Texas A\&M Energy Institute, with partial support from the National Science Foundation through the GOALI (CBET-2401564) and RETRO (CBET-2312457) awards.}
\thanks{The authors are with the Texas A\&M Energy Institute and the Artie McFerrin Department of Chemical Engineering, Texas A\&M University, College Station, TX 77843, USA
        {\tt\small lida.lamakani@tamu.edu, stratos@tamu.edu}}%
}

\begin{document}
\maketitle
\thispagestyle{empty}
\pagestyle{empty}
\raggedbottom

\begin{abstract}
Multiparametric optimal control in continuous time partitions the space of initial states into critical regions, each with its own control law involving a set of time-varying active constraints. The boundaries between two neighbouring regions are either hyperplanes, which are given in closed form by one linear equation, or curved boundaries, which must be computed. The sequence of time-varying active constraints plays a key role in defining the linear or nonlinear shape of a critical-region boundary. In this paper we show how to determine that shape prior to any computation. In particular, we show that the boundary is a hyperplane under two conditions. First, the new constraint becomes active at the same instant at every point of the boundary. Second, the durations for which the other constraints stay active are the same at every point of the boundary. Otherwise, the boundary is curved. Both conditions can be read from the sequences of active constraints of the two neighbouring regions, which gives five rules, each yielding a hyperplane or a curved boundary. We illustrate the rules on the Newell--Lee evaporator, where they classify the eight boundaries of a seven-region partition into four hyperplanes, given in closed form, and four curved boundaries, computed exactly.
\end{abstract}


\section{Introduction}

\begin{figure}[t]
\centering
\includegraphics[width=\columnwidth]{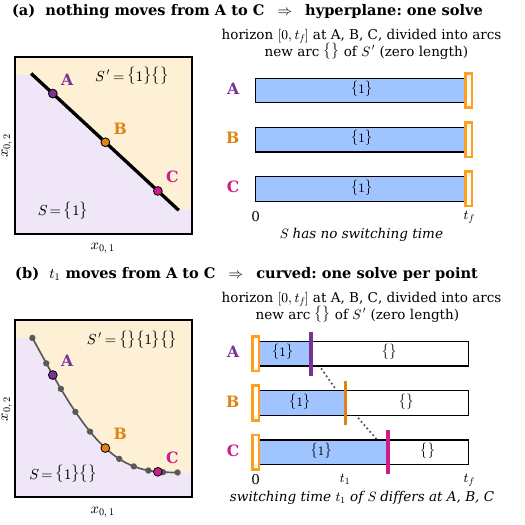}
\caption{Two neighbouring critical regions with labels $S$ and $S'$, and their boundary. Left, the space of initial states; each region is marked by its label, the arc sequence of its optimal solution, so $\{1\}\{\}$ means an arc with constraint~1 active followed by a free arc. Right, at the boundary points A, B, C, the horizon $[0,t_{f}]$ divided into the arcs (blue, constraint~1 active; white, free); the switching time $t_{1}$ is where one arc ends and the next begins; the event time is where the new arc of $S'$ appears (orange outline, of zero length on the boundary). (a) There is no switching time, so the solution has the same form at A, B and C; the boundary is a hyperplane. (b) The switching time $t_{1}$ differs at A, B and C; the boundary is curved.}
\label{fig:schematic}
\end{figure}

Multiparametric model predictive control solves the optimal control problem once, offline, for all initial states \cite{bemporad2002,pistikopoulos2020}. The result is a partition of the initial-state space into critical regions, with one control law for each region. Most of this work is done in discrete time, where the horizon is divided into $N$ intervals~\cite{bemporad2002,pistikopoulos2020,biegler2007}, the problem becomes a quadratic program, and the Karush--Kuhn--Tucker (KKT) conditions define the regions, whose number increases rapidly with $N$. Working with the continuous-time problem avoids this growth \cite{sakizlis2005,sun2016,lamakani2026a,lamakani2026c}. Pontryagin's minimum principle (PMP) provides the optimal input as a sequence of time-varying active constraints (an arc sequence)~\cite{srinivasan2003}, each arc being an interval on which the set of active constraints does not change, and a critical region is the set of initial states whose solution follows the same arc sequence, or structure; we call this sequence the \emph{label} of the region, denote it by $S$, and write it as the list of the active constraints; for example, $\{3\}\{1\}\{\}$ means that the horizon is divided into three arcs, with only constraint~3 active on the first, only constraint~1 active on the second, and no constraint active on the third; the instants at which one arc ends and the next begins are the switching times. Inside a region only the sequence is fixed; the switching times, and with them the arc durations, vary continuously with the initial state. 

\begin{figure*}[t]
\centering
\includegraphics[width=0.88\textwidth]{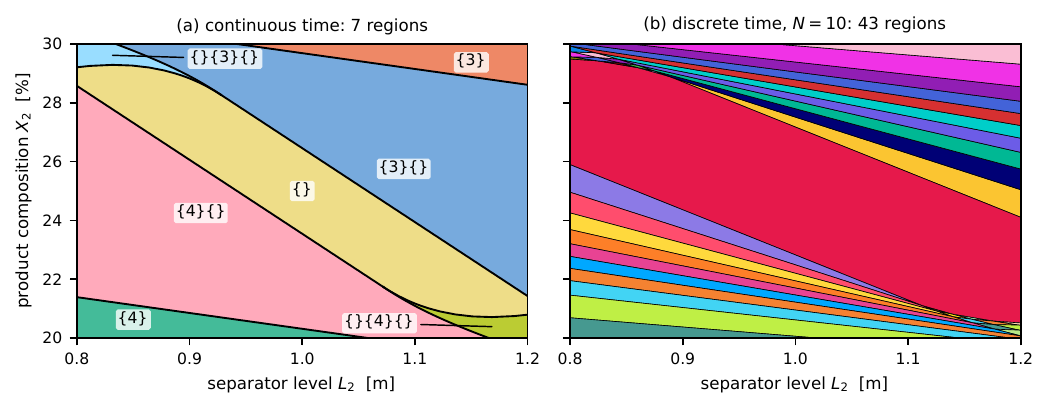}
\caption{The evaporator problem on the same parameter box. (a) Continuous time, seven regions, eight boundaries, all computed exactly. (b) Discrete time, $N=10$, as a multiparametric quadratic program, $43$ polytopes.}
\label{fig:map}
\end{figure*}

The continuous-time partition is more difficult to describe. In discrete time the regions are polytopes; in continuous time they need not be; a boundary is where an arc of the optimal input appears or disappears, the condition for this involves the switching times, and these depend nonlinearly on the initial state. In \cite{sun2016,sun2019} the boundaries are therefore left implicit and a classifier is trained to distinguish the regions. Here the boundaries are computed.

This paper presents a procedure that computes the continuous-time partition exactly, without solving an optimal control problem at any point (Section~\ref{sec:framework}, Fig.~\ref{fig:flow}), and focuses on the step in which a boundary is computed. There are two types of boundary, and the type determines the cost of that step. Fig.~\ref{fig:schematic} shows both schematically, and Fig.~\ref{fig:map}(a) corresponds to an evaporator example with four hyperplanes and four curves among its eight boundaries. The question addressed in this paper is how to test the type of a boundary before computing it. The key idea is that the labels of the two neighbouring regions are sufficient.

To address this, we proceed as follows.
\begin{enumerate}
\item Once the arc sequence and the arc durations (the lengths of the arcs) are fixed, the optimal solution at any instant depends linearly on the initial state, see Lemma~\ref{lem:affine}.
\item A boundary is where a new constraint starts becoming active, or an active one stops being active, at some instant, the event time. By fact 1, this condition is one linear equation in the initial state whose coefficients depend only on the arc sequence, the durations and the event time. The boundary is a hyperplane when the event time and the durations are the same at every point of the boundary, and curved otherwise, see Theorem~\ref{thm:main}.
\item Whether the event time and the durations are the same along the boundary can be read from the arc sequences of the two neighbouring regions, which results in five cases, each a hyperplane or curved, see Proposition~\ref{prop:test}.
\end{enumerate}
Note that the case of one input bound and a single switch was discussed in \cite{lamakani2026b}; there the durations cannot move, so the first condition alone is decisive. The present work provides a generalisation of that result, which becomes one case of the general rule.

\section{Multiparametric Optimal Control Problem}\label{sec:problem}

\subsection{Optimal control problem}

Let $x(t)\in\mathbb{R}^{n}$ be the state and $u(t)\in\mathbb{R}$ a scalar
input. The dynamics are linear and time invariant,
\begin{equation}
\dot x(t)=Ax(t)+Bu(t),\qquad x(0)=x_{0},
\label{eq:dyn}
\end{equation}
and the cost over a fixed horizon $[0,t_{f}]$ is
\begin{equation}
J(x_{0})=\tfrac12 x(t_{f})^{\!\top}\!Px(t_{f})
+\tfrac12\!\int_{0}^{t_{f}}\!\big(x^{\!\top}\!Qx+Ru^{2}\big)\,dt ,
\label{eq:cost}
\end{equation}
with $Q\succeq0$, $R>0$ and $P\succeq0$. The trajectory has to satisfy $m$ path
constraints,
\begin{equation}
g_{i}(x,u)=C_{x,i}\,x+C_{u,i}\,u-d_{i}\le 0,\qquad i=1,\dots,m,
\label{eq:con}
\end{equation}
for all $t\in[0,t_{f}]$. The initial state $x_{0}\in\Theta\subset\mathbb{R}^{n}$
is the parameter.

\begin{assumption}\label{as:mixed}
$C_{u,i}\neq 0$ for every $i$.
\end{assumption}

Input bounds are the special case $C_{x,i}=0$. Pure state constraints introduce costate jumps \cite{sun2016,hartl1995} and are outside the scope of this paper.

\subsection{Arcs and critical regions}

Introduce the costate $\lambda(t)\in\mathbb{R}^{n}$ and the Hamiltonian
\begin{equation}
\mathcal{H}=\tfrac12\big(x^{\!\top}\!Qx+Ru^{2}\big)+\lambda^{\!\top}(Ax+Bu)
+\textstyle\sum_{i}\mu_{i}\,g_{i}(x,u).
\end{equation}
Along the optimal solution the costate satisfies $\dot\lambda=-\partial\mathcal{H}/\partial x$ with the terminal condition $\lambda(t_{f})=Px(t_{f})$. $\mathcal{H}$ is quadratic in $u$ with second derivative $R>0$, hence convex in $u$; at each instant it has one minimiser over the feasible inputs, which is the optimal input, and the PMP conditions are necessary and sufficient \cite{abhijeet2024}. Where no constraint is active, $\partial\mathcal{H}/\partial u=0$ gives $u=-B^{\!\top}\!\lambda/R$. Where constraint $i$ is active, $g_{i}=0$ fixes the input, $u=(d_{i}-C_{x,i}x)/C_{u,i}$, and the same stationarity condition gives its multiplier $\mu_{i}=-(Ru+B^{\!\top}\!\lambda)/C_{u,i}$. With a scalar input at most one constraint is active on an arc; with several inputs the arc equations remain linear and only the case list of Proposition~\ref{prop:test} grows.

The horizon therefore splits into arcs, separated by the switching times. On each arc the set of active
constraints stays fixed and the pair $z=(x,\lambda)$ satisfies a linear ordinary differential equation with constant coefficients $M_{k}$ and $b_{k}$,
\begin{equation}
\dot z = M_{k}z+b_{k}.
\label{eq:arc}
\end{equation}
On an unconstrained arc $k$,
\begin{equation}
M_{k}=\begin{bmatrix}A & -BB^{\!\top}\!/R\\[2pt] -Q & -A^{\!\top}\end{bmatrix},
\qquad b_{k}=0 ,
\end{equation}
and on an arc $k$ where constraint $i$ is active, substituting the algebraic input
gives, with $A_{i}=A-BC_{x,i}/C_{u,i}$ and
$\Omega_{i}=Q+(R/C_{u,i}^{2})\,C_{x,i}^{\!\top}C_{x,i}$,
\begin{equation}
M_{k}=\begin{bmatrix}A_{i} & 0\\[2pt] -\Omega_{i} & -A_{i}^{\!\top}\end{bmatrix},
\qquad
b_{k}=\begin{bmatrix} B\,d_{i}/C_{u,i}\\[2pt]
(R\,d_{i}/C_{u,i}^{2})\,C_{x,i}^{\!\top}\end{bmatrix}.
\label{eq:Mb}
\end{equation}

\subsection{Where the boundaries come from}

A region ends where its structure is no longer optimal. Inside a region the multipliers and the constraint functions retain a strict sign along the whole trajectory. On a constrained arc the multiplier of the active constraint is positive, and wherever a constraint is inactive its function $g_{i}$ is negative. A boundary is reached when one of them touches zero at some instant $t^{*}$, which we call the \emph{event time}. One constraint becomes active or inactive there, so one arc appears. Hence two neighbouring regions differ by exactly one arc. On the boundary the new arc has zero duration, and inside the neighbouring region its duration is positive. The new arc can appear in four ways:
\begin{itemize}
\item at the start of the horizon, $t^{*}=0$, as in $\{\}\,|\,\{3\}\{\}$;
\item at the end of the horizon, $t^{*}=t_{f}$, as in $\{3\}\,|\,\{3\}\{\}$;
\item at a switching time of the structure, when a second constraint also becomes active or inactive there, as in $\{3\}\{\}\,|\,\{3\}\{1\}\{\}$;
\item inside an arc, when a constraint function touches zero without crossing it, as in $\{\}\,|\,\{\}\{3\}\{\}$; the arc is split, so the label gains two entries although only one arc is new.
\end{itemize}

Both $\mu_{i}$ and $g_{i}$ are linear in $z=(x,\lambda)$ at any fixed time, $c_{\mathrm{ev}}^{\!\top}z+c_{0}$ with $(c_{\mathrm{ev}},c_{0})$ determined by the constraint and the arc. That one of them vanishes at $t^{*}$ is the \emph{event condition}, with coefficients $(c_{\mathrm{ev}},c_{0})$. 

\section{Overview of the Partition Procedure}
\label{sec:framework}

\begin{figure}[t]
\centering
\includegraphics[width=0.84\columnwidth]{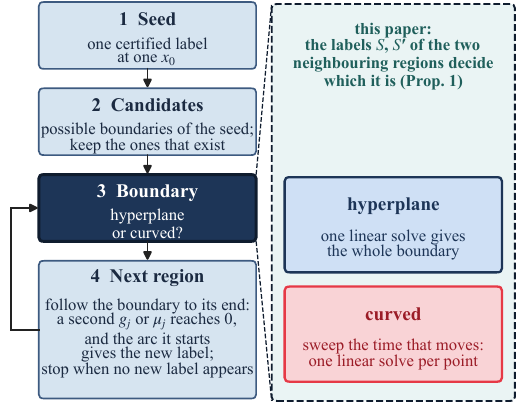}
\caption{The procedure that computes the partition, and the focus of this paper.}
\label{fig:flow}
\end{figure}

Fig.~\ref{fig:flow} shows the procedure, given in full in a companion paper. From one label verified at one initial state (Seed), the possible boundaries of its region are listed, one for each way a new arc can appear (Section~\ref{sec:problem}-C), and tested for existence (Candidates). Each existing boundary is computed (Boundary) and followed to its end, where a second constraint function or multiplier reaches zero; that condition identifies the arc that appears there, hence the next label and the next boundary to follow (Next region), until no new label appears.

Step~3 is the critical step and determines the computational cost. A hyperplane requires one linear solve, a curved boundary a sweep of the moving time, each point exact at the cost of one linear solve; which of the two a boundary is must be known before the step starts, and the remainder of this paper provides that test from the labels of the two neighbouring regions alone.

\section{One Linear Equation per Boundary}\label{sec:step1}

\subsection{Dependence of the trajectory on the initial state}

Fix an arc sequence and the arc durations $h=(h_{1},\dots,h_{K})$ with
$\sum_{k}h_{k}=t_{f}$, and let $t_{k}=h_{1}+\dots+h_{k}$, $t_{0}=0$. On arc
$k$ the pair $z$ satisfies \eqref{eq:arc} with constant coefficients, so
$z(t_{k})=\Phi_{k}z(t_{k-1})+\varphi_{k}$ with
\begin{equation}
\begin{bmatrix}\Phi_{k}&\varphi_{k}\\0&1\end{bmatrix}
=\exp\!\left(\!\begin{bmatrix}M_{k}&b_{k}\\0&0\end{bmatrix}\!h_{k}\right).
\label{eq:arcflow}
\end{equation}
Applying \eqref{eq:arcflow} arc after arc gives $z(t)=\Phi(h,t)\,z(0)+\varphi(h,t)$ for every $t$. For $t$ in arc $k$, $\Phi(h,t)=\Phi_{k}(t-t_{k-1})\,\Phi_{k-1}\cdots\Phi_{1}$, where $\Phi_{k}(t-t_{k-1})$ is \eqref{eq:arcflow} with $h_{k}$ replaced by $t-t_{k-1}$, the part of arc $k$ before $t$; $\varphi(h,t)$ is the constant term produced by the same product.
\begin{lemma}\label{lem:affine}
For a given arc sequence, the state and the costate at any time $t\in[0,t_{f}]$ are
\begin{equation}
z(t;x_{0})=\Phi(h,t)\big(W(h)x_{0}+w(h)\big)+\varphi(h,t),
\label{eq:affine}
\end{equation}
where $\Phi(h,t)$, $\varphi(h,t)$ are the products of matrix exponentials defined above, $W=\left[\begin{smallmatrix}I\\G\end{smallmatrix}\right]$ and $w=\left[\begin{smallmatrix}0\\\eta\end{smallmatrix}\right]$ with $\lambda(0)=G(h)\,x_{0}+\eta(h)$. All coefficients depend only on the arc sequence, the durations $h$ and the time $t$, and $x_{0}$ enters linearly; hence, for fixed $h$, the trajectory is affine in $x_{0}$ at every $t$.
\end{lemma}

\begin{proof}
Write $z(0)=(x_{0},\lambda_{0})$, with $\lambda_{0}$ unknown, and split the flow over the horizon, $z(t_{f})=\Phi(h,t_{f})z(0)+\varphi(h,t_{f})$, into its state and costate blocks $\Phi_{xx},\Phi_{x\lambda},\Phi_{\lambda x},\Phi_{\lambda\lambda}$ and $\varphi_{x},\varphi_{\lambda}$. The terminal condition $\lambda(t_{f})=Px(t_{f})$ then gives one linear equation for $\lambda_{0}$,
\begin{equation}
L\,\lambda_{0}=(P\Phi_{xx}-\Phi_{\lambda x})\,x_{0}+P\varphi_{x}-\varphi_{\lambda},
\qquad L=\Phi_{\lambda\lambda}-P\Phi_{x\lambda}.
\label{eq:term}
\end{equation}
$L$ is invertible; with the arcs and durations fixed, the problem is strictly convex in its only remaining unknown, the input on the free arcs, so it has exactly one solution and one $\lambda_{0}$ for every $x_{0}$. Solving \eqref{eq:term} gives $\lambda_{0}=G(h)\,x_{0}+\eta(h)$ with
\begin{equation}
G=L^{-1}(P\Phi_{xx}-\Phi_{\lambda x}),\qquad
\eta=L^{-1}(P\varphi_{x}-\varphi_{\lambda}).
\label{eq:Gg}
\end{equation}
$G$ and $\eta$ contain only $P$ and the blocks of $\Phi(h,t_{f})$, $\varphi(h,t_{f})$, built from the constant $M_{k}$, $b_{k}$ and the durations; $x_{0}$ does not appear, so $\lambda_{0}$ is affine in $x_{0}$, $z(0)=Wx_{0}+w$, and the flow gives \eqref{eq:affine}.
\end{proof}

Note that $G$ and $\eta$ depend on all the durations $h$, because $\lambda(0)$ is fixed by the terminal condition at $t_{f}$. So the costate at the event time depends on every switching time, including those after the event, and so do the coefficients of the event condition. This is the mechanism behind condition (ii) of Theorem~\ref{thm:main} and behind Fig.~\ref{fig:schematic}(b).

\subsection{The boundary as an equation in the initial state}

Consider a boundary $\mathcal{F}$. At every point of $\mathcal{F}$ the
optimal solution has the same structure $S$ with durations $h(x_{0})$, and the event condition with coefficients $(c_{\mathrm{ev}},c_{0})$ holds at the event time $t^{*}(x_{0})$. By Lemma~\ref{lem:affine},
\begin{equation}
\varepsilon(x_{0})=c_{\mathrm{ev}}^{\!\top}z\big(t^{*};x_{0}\big)+c_{0}
=a(x_{0})^{\!\top}x_{0}+a_{0}(x_{0}),
\label{eq:eps}
\end{equation}
with
\begin{equation}
a^{\!\top}=c_{\mathrm{ev}}^{\!\top}\,\Phi\,W,\qquad
a_{0}=c_{\mathrm{ev}}^{\!\top}\big(\Phi\,w+\varphi\big)+c_{0},
\label{eq:ab}
\end{equation}
where $\Phi=\Phi(h,t^{*})$, $\varphi=\varphi(h,t^{*})$, $W=W(h)$, $w=w(h)$, and
$\mathcal{F}=\{x_{0}:\varepsilon(x_{0})=0\}$. We call a boundary \emph{curved} when it is not contained in any hyperplane. The coefficients $a$ and $a_{0}$ are products of the matrix exponentials \eqref{eq:arcflow} with constant vectors, so they depend only on three quantities, the arc sequence, the durations $h$ and the event time $t^{*}$. The boundary is a hyperplane when the three remain fixed along it, which is the subject of Section~\ref{sec:step2}.

\section{When the Boundary Is a Hyperplane}\label{sec:step2}

\begin{theorem}\label{thm:main}
Let $\mathcal{F}$ be a smooth $(n-1)$-dimensional piece of boundary between two
critical regions, $n$ the dimension of $x_{0}$, defined by the event condition with coefficients $(c_{\mathrm{ev}},c_{0})$, and let Assumption~\ref{as:mixed}
hold.
\begin{itemize}
\item[(a)] $\mathcal{F}$ lies in a hyperplane if (i) the event time $t^{*}(x_{0})$ is the same for every $x_{0}\in\mathcal{F}$ and (ii) the arc durations $h(x_{0})$ are the same for every $x_{0}\in\mathcal{F}$. If $t^{*}=0$ on $\mathcal{F}$ and the event condition
involves the state only, (ii) is not needed.
\item[(b)] If the structure $S$ along $\mathcal{F}$ has no switching point,
then $h=t_{f}$ and (ii) holds automatically. Suppose also that distinct times $t$
give distinct hyperplanes $\{x_{0}:c_{\mathrm{ev}}^{\!\top}z(t;x_{0})+c_{0}=0\}$.
Then $\mathcal{F}$ lies in a hyperplane if and only if (i) holds.
\end{itemize}
\end{theorem}

\begin{proof}
(a) Take any point $x_{0}$ of $\mathcal{F}$. By \eqref{eq:eps} it satisfies
$a(x_{0})^{\!\top}x_{0}+a_{0}(x_{0})=0$, and by \eqref{eq:ab} the two
coefficients depend on $x_{0}$ only through $t^{*}$ and $h$. Under (i) and
(ii), $t^{*}$ and $h$ are constant along $\mathcal{F}$, so $a$ and $a_{0}$ are
constant along $\mathcal{F}$, and every $x_{0}\in\mathcal{F}$ satisfies one
affine equation $a^{\!\top}x_{0}+a_{0}=0$, a hyperplane. For the special case, $t^{*}=0$ gives $\Phi=I$ and $\varphi=0$ in \eqref{eq:ab}, so $a^{\!\top}=c_{\mathrm{ev}}^{\!\top}W$ and $a_{0}=c_{\mathrm{ev}}^{\!\top}w+c_{0}$; an event that involves the state only has $c_{\mathrm{ev}}=(c_{x},0)$, so only the state rows of $W$ and $w$ enter, which are $I$ and $0$ whatever $h$ is, and $a=c_{x}$, $a_{0}=c_{0}$ contain no $h$.

(b) With no switching point, $h=t_{f}$ cannot move, so (ii) holds and only (i) remains. With $h$ fixed, \eqref{eq:affine} makes the event function affine in $x_{0}$ with coefficients that depend on $t$ only, $\psi(t,x_{0})=c_{\mathrm{ev}}^{\!\top}z(t;x_{0})+c_{0}=\alpha(t)^{\!\top}x_{0}+\beta(t)$. The region is $C=\{x_{0}:\psi(t,x_{0})\le0 \text{ for all } t\}$ ($\ge0$ for a multiplier). For each fixed $t$, $\psi(t,x_{0})\le0$ is one linear inequality in $x_{0}$, that is, one side of a hyperplane, and $C$ is the set of $x_{0}$ that satisfy all of them at once, which is convex. Its boundary $\mathcal{F}$ is where $\max_{t}\psi(t,x_{0})=0$, reached at $t^{*}(x_{0})$. If (i) holds, $t^{*}=\tau$ on $\mathcal{F}$, so $\mathcal{F}$ lies in the hyperplane $H(\tau)=\{x_{0}:\psi(\tau,x_{0})=0\}$. Conversely, suppose $\mathcal{F}$ lies in one hyperplane $H$. Take $\bar{x}_{0}\in\mathcal{F}$ and $\tau=t^{*}(\bar{x}_{0})$. Then $\psi(\tau,\cdot)\le0$ on $C$, hence on $\mathcal{F}$, and $\psi(\tau,\bar{x}_{0})=0$, so the linear function $\psi(\tau,\cdot)$ reaches its largest value on the flat set $\mathcal{F}$ at an interior point; a linear function can do so only if it is constant, so $\psi(\tau,\cdot)=0$ on $\mathcal{F}$ and $\mathcal{F}\subset H(\tau)=H$. This holds at every point of $\mathcal{F}$, and distinct times give distinct hyperplanes, so $t^{*}$ is the same at every point, which is (i).
\end{proof}

In general, condition (i) alone is not sufficient; Section~\ref{sec:curved} shows a boundary whose event is fixed at $t=0$ and which is curved.

\begin{figure}[!t]
\centering
\includegraphics[width=\columnwidth]{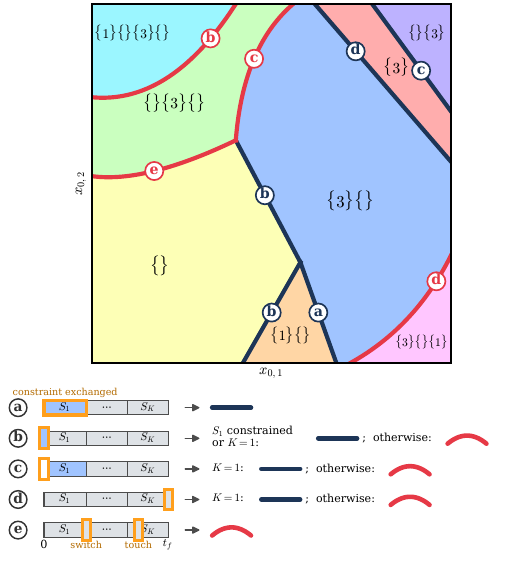}
\caption{The five cases of Proposition~\ref{prop:test} on a schematic partition of the initial-state space (illustrative, not computed). Each boundary carries its case letter; hyperplanes in navy, curved boundaries in red. Cases (b), (c) and (d) appear once straight and once curved, depending on the structure $S$ on the boundary, the label with fewer arcs. Bottom, for each case: the bar is the horizon of $S$ with its arcs $S_{1},\dots,S_{K}$ (blue, a constrained arc; white, a free arc; grey, an arc of either kind); the orange outline is the new arc of the neighbouring region, filled by its kind; the stroke that follows is the resulting shape of the boundary.}
\label{fig:cases}
\end{figure}

\section{The Two Conditions from the Labels}\label{sec:step3}

Theorem~\ref{thm:main} needs $t^{*}$ and $h$, which are known only after a trajectory has been solved. The labels provide both without any computation; the labels of two neighbouring regions differ by the new arc. Removing it yields the structure $S$ on the boundary, and its position gives the event time; a new arc placed first corresponds to $t^{*}=0$, placed last to $t^{*}=t_{f}$, and placed between two arcs to an interior instant. The list of cases is complete because the new arc can appear only in the four ways of Section~\ref{sec:problem}-C and is either constrained or free; the exchange of constraints at $t=0$, case (a), adds no arc. Fig.~\ref{fig:cases} shows the five cases on a schematic partition.

\begin{proposition}\label{prop:test}
Let two neighbouring critical regions have arc sequences that differ by one
event, and let $S$ be the one with fewer arcs, the structure of the solution on the
boundary, with $K$ arcs and first arc $S_{1}$. In case (a) both sequences have
$K$ arcs and either can be taken as $S$. Under Assumption~\ref{as:mixed} the
boundary between the two regions is classified by the position of the new arc.
\begin{itemize}
\item[(a)] No new arc; only the constraint active on the first arc differs. \textbf{Hyperplane}.
\item[(b)] A constrained arc is added before the first arc of $S$. \textbf{Hyperplane} if $S_{1}$ is constrained or $K=1$; otherwise \textbf{curved}.
\item[(c)] A free arc is added before the first arc of $S$. \textbf{Hyperplane} if $K=1$; otherwise \textbf{curved}.
\item[(d)] An arc is added after the last arc of $S$. \textbf{Hyperplane} if $K=1$; otherwise \textbf{curved}.
\item[(e)] An arc is added at an interior switching time of $S$, or an arc of $S$ is split by the new arc. \textbf{Curved}.
\end{itemize}
The hyperplane cases are proved by Theorem~\ref{thm:main}(a); the curved cases assume that the moving coefficients $a$, $a_{0}$ in \eqref{eq:ab} do not cancel, except case (e) with $K=1$, which is exact by Theorem~\ref{thm:main}(b).
\end{proposition}

\begin{proof}
The new arc occupies the single instant $t^{*}$, and the event condition is \eqref{eq:eps} evaluated there. Case (a): both constraints hold at $t=0$; eliminating $u(0)$ between $g_{i}=0$ and $g_{j}=0$ gives $(C_{x,j}-C_{u,j}C_{x,i}/C_{u,i})\,x_{0}=d_{j}-C_{u,j}d_{i}/C_{u,i}$, whose coefficients are numbers of the problem, the same at every point, a hyperplane for any $h$. Cases (b) and (c): $t^{*}=0$ and the event is $g_{j}(x_{0},u(0))=0$, respectively $\mu_{i}(0)=0$. If $S_{1}$ is constrained by $i$, $u(0)=(d_{i}-C_{x,i}x_{0})/C_{u,i}$ contains the state only and Theorem~\ref{thm:main}(a) gives a hyperplane for any $K$; otherwise the event contains $\lambda(0)=G(h)x_{0}+\eta(h)$, which is fixed for $K=1$, a hyperplane, and moves with $h$ for $K\ge2$, curved. Case (d): $t^{*}=t_{f}$ and the event contains $\Phi(h,t_{f})$, fixed for $K=1$, a hyperplane, and moving for $K\ge2$, curved. Case (e): $t^{*}$ is a switching time of $S$ or the instant of a touch, both determined by the trajectory from $x_{0}$ and in general moving with it, so (i) fails; curved, except when that instant stays constant along the boundary.
\end{proof}

In summary, the boundary is curved whenever a time can move along it and is a hyperplane when no quantity entering the event condition can move.

\begin{figure}[t]
\centering
\includegraphics[width=0.66\columnwidth]{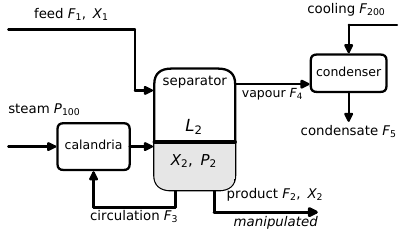}
\caption{The forced-circulation evaporator, states $L_{2}$ (level), $X_{2}$ (composition) and input $F_{2}$ (product flow).}
\label{fig:schem}
\end{figure}

\section{Forced-Circulation Evaporator Example}\label{sec:evap}

\subsection{Model}

We use the Newell--Lee forced-circulation evaporator \cite{newell1989,sakizlis2003}, a standard benchmark that concentrates a dilute feed by boiling off solvent (Fig.~\ref{fig:schem}). With the other inputs at their nominal values and the pressure regulated, the model reduces to the separator level $L_{2}$ and the product composition $X_{2}$, with the product flow $F_{2}$ as the single input. The mixed path constraint is the solute draw rate $F_{2}X_{2}$, in which input and state both appear (Assumption~\ref{as:mixed}).

We work in scaled deviations, $x_{1}=(L_{2}-1.0)/0.1$ ($L_{2}$ in m), $x_{2}=(X_{2}-25)/2.5$ ($X_{2}$ in \%) and $u=(F_{2}-2.0)/0.5$ ($F_{2}$ in kg/min), so the box $\Theta=[-2,2]^{2}$ is $L_{2}\in[0.8,1.2]$~m and $X_{2}\in[20,30]\%$. Linearising at the nominal point gives
\begin{equation}
A=\begin{bmatrix}0 & 0.104538\\ 0 & -0.1\end{bmatrix},\quad
B=\begin{bmatrix}-0.25\\ -0.25\end{bmatrix},
\label{eq:AB}
\end{equation}
in units of $\mathrm{min}^{-1}$. We take $t_{f}=10$~min, $Q=I$, $R=1$, $P=I$.

The four path constraints are the flow limits $F_{2}\le3.0$ and $F_{2}\ge1.0$~kg/min (constraints 1, 2) and the draw-rate limits $F_{2}X_{2}\le57$ and $F_{2}X_{2}\ge43$ (constraints 3, 4), linearised at the nominal point; in scaled variables,
\begin{equation}
C_{x}\!=\!\begin{bmatrix}0&0\\0&0\\0&5\\0&-5\end{bmatrix}\!,\;\;
C_{u}\!=\!\begin{bmatrix}1\\-1\\12.5\\-12.5\end{bmatrix}\!,\;\;
d\!=\!\begin{bmatrix}2\\2\\7\\7\end{bmatrix}\!.
\end{equation}
The flow limits are never active in the box.

\subsection{Partition}

The partition of the box has seven critical regions and eight boundaries,
Fig.~\ref{fig:map}(a). It was computed by the procedure of Section~\ref{sec:framework}, the four hyperplanes by one linear solve each and the four curved boundaries by sweeping the moving quantity of each; no optimality system was solved.
Table~\ref{tab:faces} lists the four boundaries of constraint~3 with their case, chord (distance between the endpoints, scaled units) and sagitta (largest distance from the boundary to the chord, zero for a hyperplane). Constraint~4 mirrors constraint~3 (opposite signs in $C_{x}$ and $C_{u}$, same $d$), so its four boundaries are their mirror images and are not listed.

\begin{table}[t]
\caption{Boundaries of constraint~3 in the evaporator.}
\label{tab:faces}
\centering
\footnotesize
\setlength{\tabcolsep}{1.4pt}
\begin{tabular}{@{}lllllc@{}}
\toprule
Case & Boundary & Type & $t^{*}$ & Moves & Sagitta, chord \\
\midrule
(b) & $\{\}|\{3\}\{\}$ & hyperplane & $0$ & none & $0$, line \eqref{eq:cf1} \\
(d) & $\{3\}|\{3\}\{\}$ & hyperplane & $t_{f}$ & none & $0$, line \eqref{eq:cf2} \\
(c) & $\{3\}\{\}|\{\}\{3\}\{\}$ & curved & $0$ & $h_{1}$, 0--7.7 min & $0.042$, $1.16$ \\
(e) & $\{\}|\{\}\{3\}\{\}$ & curved & 0--5.0 min & $t^{*}$ & $0.177$, $1.33$ \\
\bottomrule
\end{tabular}
\end{table}

\subsection{The hyperplanes}
\label{sec:lines}

Here cases (b) and (d) have $K=1$, so \eqref{eq:ab} gives the boundary in
closed form. For the boundary between $\{\}$ and $\{3\}\{\}$, where constraint~3 becomes active at the start, the result is, with $L_{2}$ in m and $X_{2}$ in \%,
\begin{equation}
25.15\,L_{2} + X_{2} = 51.61 ,
\label{eq:cf1}
\end{equation}
and for the boundary between $\{3\}$ and $\{3\}\{\}$, where constraint~3 becomes inactive at the end,
\begin{equation}
5.38\,L_{2} + X_{2} = 35.07 .
\label{eq:cf2}
\end{equation}
\subsection{The curved boundaries, computed exactly}
\label{sec:curved}

On a boundary of case (c) or (e) exactly one quantity moves, and the boundary is computed from the structure with fewer arcs. In case (c), between $\{3\}\{\}$ and $\{\}\{3\}\{\}$, take the solution with structure $\{3\}\{\}$, constraint~3 active for a duration $h_{1}$ and then free, and fix $h_{1}$. By Lemma~\ref{lem:affine} the multiplier $\mu_{3}$ at any instant is then affine in $x_{0}$, and two conditions hold on the boundary, the switching condition $\mu_{3}(h_{1})=0$ of $\{3\}\{\}$ and the event $\mu_{3}(0)=0$, which lets the free arc of $\{\}\{3\}\{\}$ begin at $t=0$. Solving these two affine equations in the two unknowns $x_{0}$ gives one boundary point for this $h_{1}$. In case (e), between $\{\}$ and $\{\}\{3\}\{\}$, take the free solution, structure $\{\}$, and fix the instant $t^{*}$ at which constraint~3 touches zero. On the boundary $g_{3}(t^{*})=0$ and $\partial g_{3}/\partial t\,(t^{*})=0$, since a touch is a maximum of $g_{3}$; by Lemma~\ref{lem:affine} these are two affine equations in $x_{0}$, and solving them gives one boundary point for this $t^{*}$. Stepping $h_{1}$, respectively $t^{*}$, over its range sweeps the boundary; every point is checked against all constraints and multipliers over the horizon, and the boundary ends where a second one reaches zero.

The case (e) boundary is visibly curved (Table~\ref{tab:faces}); since its structure $\{\}$ has no switching point, its classification is exact by Theorem~\ref{thm:main}(b). The case (c) boundary is the one that needs condition (ii); its event $\mu_{3}(0)=0$ is fixed at $t=0$, and it is curved only because $h_{1}$ moves and enters the costate at $t=0$ through the terminal condition, the situation of Fig.~\ref{fig:schematic}(b). Its sagitta is not visible in Fig.~\ref{fig:map}(a) but is nonzero, every point being obtained from two linear equations, and Proposition~\ref{prop:test} classifies it from the labels alone, without any threshold.

\subsection{What a time grid does to a curved boundary}

With zero-order hold over $N$ intervals and the constraints enforced at the nodes, the same problem is a multiparametric quadratic program, solved here with PPOPT \cite{kenefake2022}, $43$ polytopic regions at $N=10$ against seven in continuous time, Fig.~\ref{fig:map}(b). A polytope has no curved face, so discrete time replaces the curve by a sequence of narrow polytopes, and a finer grid only increases their number.

\section{Conclusion}

A critical-region boundary is one linear equation in the initial state once the arc sequence and the durations are fixed. It is a hyperplane when the coefficients of that equation are the same along it, that is, when the event occurs at a fixed time and the durations do not move; the second condition is the reason an event fixed at an end of the horizon is not sufficient by itself. Both conditions can be read from the arc sequences of the two neighbouring regions, so in the procedure of Section~\ref{sec:framework} a boundary is classified before it is computed, and its cost, one linear solve or a sweep, is known in advance. Two questions remain open. For several arcs the curved classification assumes that the moving coefficients do not cancel, an assumption Theorem~\ref{thm:main}(b) removes only when the boundary has no switching point; and pure state constraints, where the costate can jump, are outside the present class.

\end{document}